\documentclass[12pt]{amsart}

\usepackage{geometry}
\usepackage[all]{xy}
\usepackage{graphicx}
\usepackage{amssymb}
\usepackage{amsfonts}
\usepackage{epstopdf}
\usepackage{amsmath, amsthm}
\usepackage{mathrsfs}
\usepackage{rotating}
\usepackage{multirow}
\numberwithin{equation}{section}

\newcommand{\bb}[1]{\mathbb{#1}}
\newcommand{\cc}[1]{\mathcal{#1}}
\newcommand{\lie}[1]{\mathfrak{#1}}

\newcommand{\diag}{\textrm{diag}}
\def\inv{^{-1}}

\theoremstyle{plain}
\newtheorem{theorem}{Theorem}[section]

\newtheorem{lemma}[theorem]{Lemma}
\newtheorem{proposition}[theorem]{Proposition}

\theoremstyle{definition}
\newtheorem{example}[theorem]{Example}
\newtheorem{remark}[theorem]{Remark}

\newtheorem{definition}[theorem]{Definition}

\begin{document}
\title[Lower bounds for Gromov width of coadjoint orbits in U(n).]
{Lower bounds for Gromov width of coadjoint orbits in U(n).}

\author{Milena Pabiniak}

\address{Milena Pabiniak,
Department of Mathematics,  
Cornell University, 
Ithaca NY}
\email{milena@math.cornell.edu}

\thanks{\today}

\maketitle

\begin{abstract} We use the Gelfand-Tsetlin pattern to construct an effective Hamiltonian, completely integrable action of a torus $T$
on an open dense subset of a coadjoint orbit of the unitary group. We then identify a proper Hamiltonian $T$-manifold centered around a point in 
the dual of the Lie algebra of $T$.
A theorem of 
Karshon and Tolman says that such a manifold is equivariantly symplectomorphic to a particular subset of $\bb{R}^{2D}$.
This fact enables us to construct symplectic embeddings of balls into certain coadjoint orbits of the unitary group,
and therefore obtain a lower bound for their Gromov width. 
Using the identification of the dual of the Lie algebra of the unitary group with the space of $n \times n$ Hermitian matrices, the main theorem states 
that for a coadjoint orbit through $\lambda=$diag$(\lambda_1, \ldots , \lambda_n)$ in the dual of the Lie algebra of the unitary group, 
where at most one eigenvalue is repeated, the lower bound for Gromov width is equal to the minimum of the differences $\lambda_i-\lambda_j$, over all $\lambda_i>\lambda_j$.
For a generic orbit (i.e. with distinct $\lambda_i$'s), with additional integrality conditions, this minimum has been proved to be exactly the Gromov width of the orbit.
For nongeneric orbits this lower bound is new.
\end{abstract}
\tableofcontents
%-------------------------------
\section{Introduction} In 1985 Mikhail Gromov proved the nonsqueezing theorem which is one of the foundational results in the modern theory of symplectic invariants.
The theorem says that a ball
$B^{2N}(r)$ of radius $r$, in a symplectic vector space $\bb{R}^{2N}$ with the usual symplectic structure, cannot be symplectically embedded into $B^2(R)\times \bb{R}^{2N-2}$
unless $r\leq R$. This motivated the definition of the invariant called the Gromov width.
Consider the ball of capacity $a$
$$ B^{2N}_a = \Big \{ z \in \bb{C}^N \ \Big | \ \pi \sum_{i=1}^N |z_i|^2 < a \Big \} , $$
with the standard symplectic form
$\omega_{std} = \sum dx_j \wedge dy_j$.
The \textbf{Gromov width} of a $2N$-dimensional symplectic manifold $(M,\omega)$
is the supremum of the set of $a$'s such that $B^{2N}_a$ can be symplectically
embedded in $(M,\omega)$. 
Equivalently, it is $
\sup\{\pi r^2\,|\, B^{2N}(r) \textrm{ can be symplectically embedded into } (M, \omega)\}.$\\
\indent In this paper we consider coadjoint orbits of $U(n)$.
Multiplying by a factor of $i$, we can identify the Lie algebra $\lie{u}(n)$ with the space of Hermitian matrices.
The pairing in $\lie{u}(n)$
$$(A,B)=\textrm{trace}(AB)$$
gives us the identification of $\lie{u}^*(n)$ with $\lie{u}(n)$. From now on, we will identify $\lie{u}^*(n)$ with the space of Hermitian matrices.\\
\indent Given a Hamiltonian torus action one can construct embeddings of balls using the information from the moment polytope.
Using this technique we prove the following theorem.
\begin{theorem}\label{main}
Consider the $U(n)$ coadjoint orbit $M:=\mathcal{O}_{\lambda}$ in $\lie{u}(n)^*$ through a point $diag\,(\lambda_1, \lambda_2, \ldots, \lambda_n)$ 
where $$\lambda_1 > \lambda_2> \ldots > \lambda_l=\lambda_{l+1}= \ldots = \lambda_{l+s}>\lambda_{l+s+1}> \ldots > \lambda_n, \,\,s\geq 0.$$ 
The Gromov width of $M$ is at least the minimum $\min\{\lambda_i-\lambda_j\,|\, \lambda_i>\lambda_j\, \}.$
\end{theorem} \noindent
In fact we prove a stronger (but more cumbersome to state) result - see Remark \ref{moresmoothness}.
\\ \indent There are reasons to care about this particular lower bound.
In the case of generic coadjoint orbits, i.e. when $\lambda_1 > \lambda_2 \ldots > \lambda_n$, Masrour Zoghi in \cite{Z}
had already obtained this lower bound. Moreover, with some additional integrability assumption on $\lambda$, he 
proved that this lower bound is precisely the Gromov width. He also proved a similar upper bound for Gromov width of generic coadjoint orbits (with some integrality conditions) 
of other simple compact Lie groups.
This suggests that the lower bound for non-generic orbits that we provide here may in fact be the Gromov width.
\\ \indent To prove the Theorem \ref{main} we will recall an action of the Gelfand-Tsetlin torus on an open dense subset of $\mathcal{O}_{\lambda}$. We will then use
the theorem of Karshon and Tolman, \cite{KT}, recalled here as Proposition \ref{embed}, to obtain symplectic embeddings of balls. Masrour Zoghi also used the Karshon and Tolman's result, 
but applied to the standard coadjoint action of a maximal torus. 
He suggested that maybe the action of the Gelfand-Tsetlin torus could give stronger results for a wider class of orbits.
\\ \indent \textbf{Organization}. Section \ref{Preliminaries} provides background about centered actions and Gelfand-Tsetlin functions.
In Section \ref{actionsection}, we carefully analyze Gelfand-Tsetlin functions and the action they induce.
Section \ref{mainproof} contains the proof of the main result. Section \ref{summary} has a ``bookkeeping'' character.
There we summarize what is known about the Gromov width of $U(n)$ coadjoint orbits for small values of $n$.
\\ \indent \textbf{Acknowledgments.} The author is very grateful to Yael Karshon for suggesting this problem and helpful conversations during my work on this project. 
The author also would like to thank her advisor, Tara Holm, for useful discussions.
%----------------------------------------------------------------------------------------------------------
\section{Preliminaries}\label{Preliminaries}
\subsection{Centered actions}
Centered actions were introduced in \cite{KT}; we include the details here for completeness and to set notation. Let $(M, \omega)$ be a connected symplectic manifold, equipped with an effective, symplectic action of a torus $T\cong (S^1)^{\dim T}.$
The action of $T$ is called \textbf{Hamiltonian} if there exists a $T$-invariant map
$\Phi \colon M \to \lie{t}^*$, called the \textbf{moment map}, such that
\begin{equation} \label{def moment}
        \iota(\xi_M) \omega = - d \left< \Phi,\xi \right>
\quad \forall \ \xi \in \lie{t},
\end{equation}
where $\xi_M$ is the vector field on $M$ generated by $\xi \in \lie{t}$.
We will identify Lie$(S^1)$ with $\bb{R}$ using the convention that the exponential map $exp:\bb{R}\cong$Lie$(S^1)\rightarrow S^1$ is
given by $t\rightarrow e^{2\pi i t},$ that is $S^1\cong \bb{R}/\bb{Z}$.\\
At a fixed point $p\in M^T$, we may consider the induced action of $T$ on the tangent space $T_pM$. There exist $\eta_j \in \lie{t}^*$,
called the \textbf{isotropy weights} at $p$, such that this action is isomorphic to the action on $(\bb{C}^n,\omega_{std})$
generated by the moment map $$ \Phi_{\bb{C}^n}(z) = \Phi(p) + \pi \sum |z_j|^2 \eta_j .$$
The isotropy weights are uniquely determined up to permutation.
By the equivariant Darboux theorem,
a neighborhood of $p$ in $M$ is equivariantly symplectomorphic
to a neighborhood of $0$ in $\bb{C}^n$.
However, this theorem does not tell us how large we may take this neighborhood to be.
Let $\cc{T} \subset \lie{t}^*$ be an open convex set which contains $\Phi(M)$.
The quadruple $(M,\omega,\Phi,\cc{T})$ is a
\textbf{proper Hamiltonian $\mathbf{T}$-manifold} if 
$\Phi$ is proper as a map to $\cc{T}$, that is,
the preimage of every compact subset of $\cc{T}$ is compact.

For any subgroup $K$ of $T$, let
$M^K = \{ m \in M \mid a \cdot m = m \ \forall a \in K \}$
denote its fixed point set.

\begin{definition} \label{centered-definition}
A proper Hamiltonian $T$-manifold $(M,\omega,\Phi,\cc{T})$
is \textbf{centered} about a point $\alpha \in \cc{T}$ if
$\alpha$ is contained in the moment map image of every component
of $M^K$, for every subgroup $K \subseteq T$.
\end{definition}
We now quote several examples and non-examples, following \cite{KT}.
\begin{example}
A compact symplectic manifold with a non-trivial $T$-action is never
centered, because it has fixed points with different moment map images.
\end{example}

\begin{example}
Let a torus $T$ act linearly on $\bb{C}^n$ with a proper moment map
$\Phi_{\bb{C}^n}$ such that $\Phi_{\bb{C}^n}(0) = 0$.  
Let $\cc{T} \subset \lie{t}^*$ be an open convex subset containing the origin. 
Then $\Phi_{\bb{C}^n}^{-1}(\cc{T})$ is centered about the origin.
\end{example}
A Hamiltonian $T$ action on $M$ is called \textbf{toric} if 
$\dim T= \frac 1 2 \dim M.$
\begin{example}\label{largest}
Let $M$ be a compact symplectic toric manifold with moment map
$\Phi \colon M \to \lie{t}^*$.   Then $\Delta := \textrm{Im } \Phi$
is a convex polytope.  The orbit type strata in $M$ are the
moment map pre-images of the relative interiors of the faces
of $\Delta$.  Hence, for any $\alpha \in \Delta$,
$$ \bigcup\limits_{\substack{F \text{ face of } \Delta \\ \alpha \in F}}
   \Phi^{-1}(\text{rel-int } F) $$
is the largest subset of $M$ that is centered about $\alpha$.
\end{example}
When the dimension of the torus acting on a compact symplectic manifold is less then half of the dimension of the manifold,
one can easily find a centered region from an x-ray of the Hamiltonian $T$-space $M$.
The \textbf{x-ray} of $(M,\omega,\phi)$ is the 
collection of convex polytopes 
 $\phi(X)$ over all connected compontents $X$ of $M^K$ for some subtorus $K$ of $T$ (for more details see \cite{To}). For the toric symplectic manifold, an x-ray is exactly the collection of faces of convex polytope that is the image of moment map.
Figure \ref{centered} presents some examples of centered regions, that we can see directly from the x-rays of $M$.
\begin{figure}[h!]\label{centered}
 \includegraphics[width=.6\textwidth]{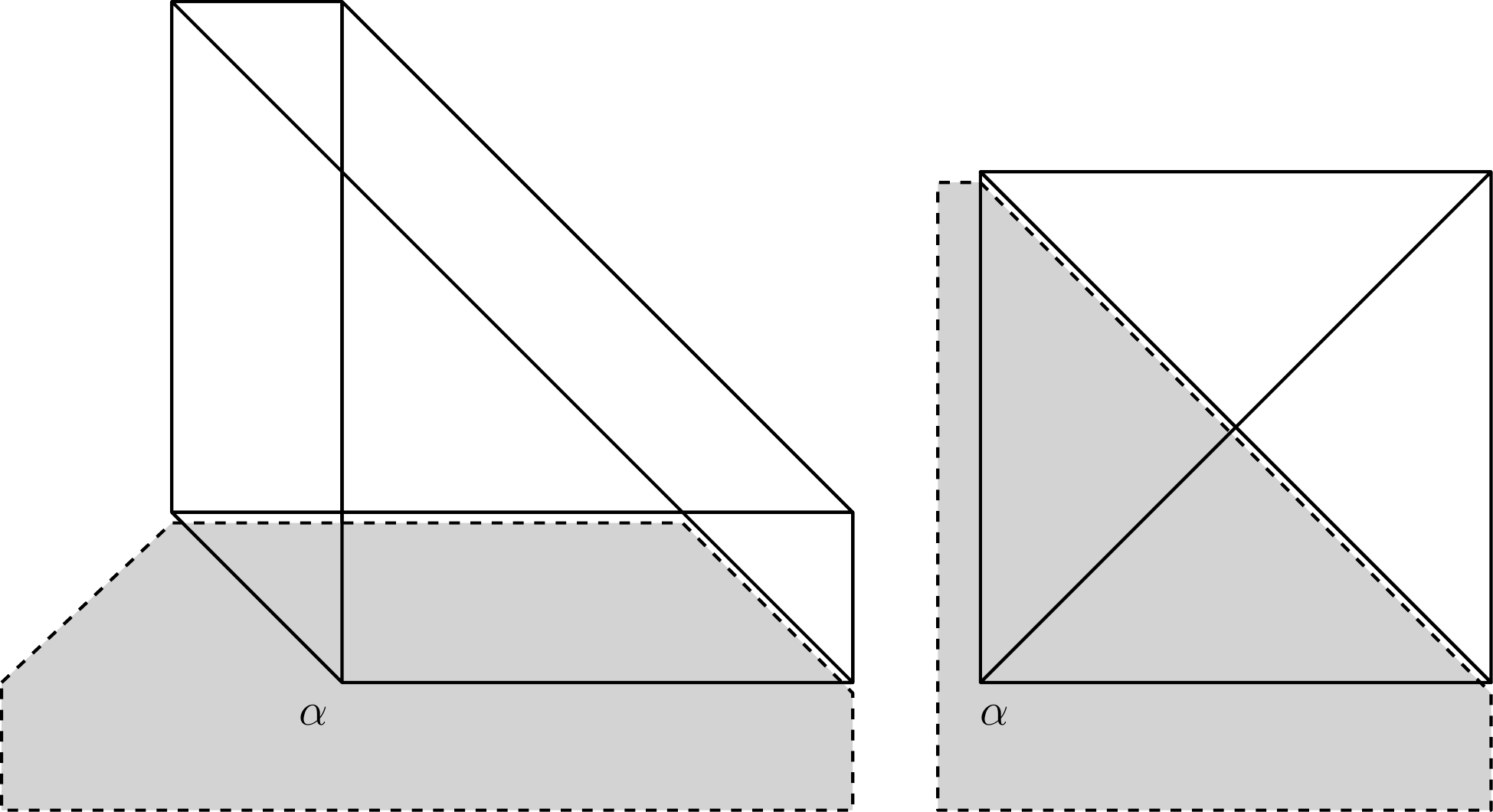}
\caption{The regions centered around $\alpha$.}
\end{figure}
\begin{example}
Let $(M,\omega,\Phi,\cc{T})$ be a proper Hamiltonian $T$-manifold.
Then every point in $\lie{t}^*$ has a neighborhood whose preimage is centered.
This is a consequence of the local normal form theorem and the properness
of the moment map.
\end{example}

\begin{proposition} (Karshon, Tolman, \cite{KT})\label{embed}
Let $(M,\omega,\Phi,\cc{T})$ be a proper Hamiltonian $T$-manifold.
Assume that $M$ is centered about $\alpha \in \cc{T}$ and that 
$\Phi^{-1}(\{\alpha\})$ consists of a single fixed point $p$.
Then $M$ is equivariantly symplectomorphic to 
$$\left\{ z \in \bb{C}^n \ | \ \alpha + \pi \sum |z_j|^2 \eta_j \in \cc{T} \right\},$$
where $\eta_1,\ldots,\eta_n$ are the isotropy weights at $p$.
\end{proposition}
\begin{example}
Consider a compact symplectic toric manifold $M$ with the following moment map image.
\begin{center}
\includegraphics[width=0.6\textwidth]{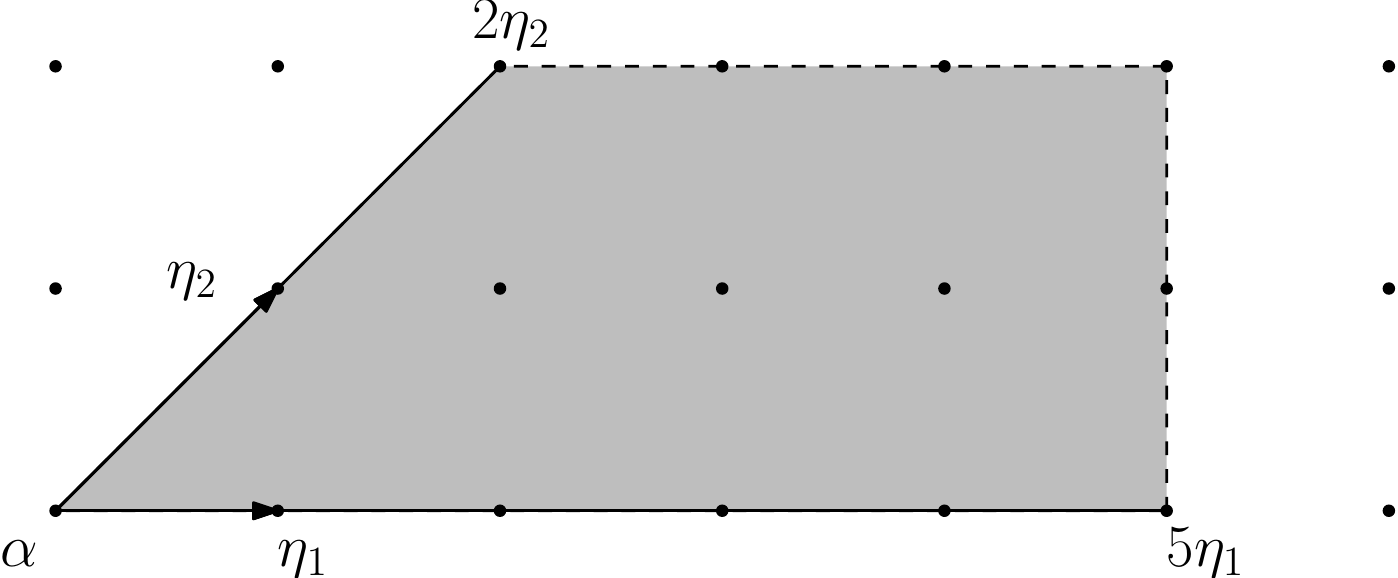}
\end{center}
The weights of the torus action are $\eta_1$ and $\eta_2$, and the lattice lengths of edges starting from $\alpha$ are $5$ and $2$
(with respect to weight lattice).
The largest subset of $M$ that is centered about $\alpha$, as described in Example \ref{largest}, maps under the moment map to the shaded region.\\
The above Proposition tells us that it is
equivariantly symplectomorphic to
$$\{z \in \bb{C}^2| \alpha + \pi (|z_1|^2 \eta_1+|z_2|^2\eta_2) \in \textrm{ shaded region } \}.$$
If $z\in B^{2}_2 =  \{ z \in \bb{C}^2 \ \Big |  \pi (|z_1|^2 \eta_1 +|z_2|^2 \eta_2) <2\}$
then $\alpha + \pi (|z_1|^2 \eta_1 +|z_2|^2 \eta_2)$ is in the shaded region.
Therefore the ball $B^{2}_2$ of capacity $2$ embeds into $M$ and the Gromov width of $M$ is at least the minimum of lattice lengths of edges of the moment polytope, starting at $\alpha$.
 \end{example}
%-----------------------------------------
\subsection{Standard torus action on a coadjoint orbit.}\label{standard action} 
Under our identifications, the coadjoint action of $U(n)$ on $\lie{u}(n)^*$ is by conjugation: $A\cdot \xi=A \xi A\inv$. 
restricted to an orbit $\cc{O}_{\lambda}$, this action is Hamiltonian with moment map inclusion $\cc{O}_{\lambda} \hookrightarrow \lie{u}(n)^*$.
Let $T$ be maximal torus in $U(n)$.
As explained in the introduction, we identify  
$\lie{u}(n)^*$ with the space of $n \times n$ Hermitian matrices. We will use coordinates $\{e_{ij}\}$, with $e_{ij}$ correspondig to $(i,j)$-th entry of a matrix.
We choose the positive Weyl chamber, $(\lie{t}^*)_+$, 
to be 
$$(\lie{t}^*)_+:=\{ \diag( \lambda_{11}, \lambda_{22}, \ldots , \lambda_{nn});\,  \lambda_{11} \geq \lambda_{22} \geq \ldots  \geq \lambda_{nn}\}.$$
Then $\Delta=\{e_{ii}-e_{jj}\,|\, i \neq j\}$ is a root system and $\Sigma =\{e_{ii}-e_{i+1,i+1}\,|\, i=1,2,\ldots, n-1\}$ is the set of positive roots.
The coadjoint orbits in $\lie{u}(n)^*$ are in one-to-one correspondence with the points of $(\lie{t}^*)_+$.
Precisely, for any $(\lambda_{11},\lambda_{22}, \ldots , \lambda_{nn}) \in (\lie{t}^*)_+$ the corresponding coadjoint orbit is the set of 
all Hermitian matrices with eigenvalues $(\lambda_{11},\lambda_{22}, \ldots , \lambda_{nn})$.
Fix some $\lambda= (\lambda_{11} \geq \lambda_{22} \geq \ldots  \geq \lambda_{nn} ) \in (\lie{t}^*)_+$ and denote by $\cc{O}_{\lambda}$
the coadjoint orbit through $\lambda$. The standard $T^n$ action on $\cc{O}_{\lambda}$ is the action of the maximal torus $T^n \subset U(n)$. The fixed points of this action are the diagonal matrices.
In particular, $\lambda$ is a fixed point and the weights of $T^n$ action on $T_{\lambda}\cc{O}_{\lambda}$ are given by the negative 
roots $-\Sigma$.
The $T^n$ action is Hamiltonian with moment map $\mu: \cc{O}_{\lambda} \rightarrow (\lie{t}^n)^* \cong \bb{R}^n$ that maps a matrix $A=(a_{ij})$ 
to the diagonal $n \times n$ matrix $\diag\,(a_{11}, \ldots, a_{nn})$.
For any $j=1, \ldots, n$, we have a natural embedding $\iota_{j}:U(j)\rightarrow U(n)$ 
\begin{equation*}
 \iota_{j}(B)= \left(\begin{array}{c|cc}
    B  & 0& \\
\hline
0&  Id &   \\
    \end{array}
\right),
\end{equation*}
where $B \in U(j)$. Using this embedding we obtain a $U(j)$ (and also $T^{j}$) action 
on $\cc{O}_{\lambda}$: for $B \in U(j)$ and $\xi \in \cc{O}_{\lambda}$, we define $$B \cdot \xi = \iota_{j}(B)\, \xi\, (\iota_{j}(B))^{-1}.$$
To simplyfy the notation, we will often write $B$ instead of $\iota_{j}(B)$.
Both of these actions are also Hamiltonian.
The moment map for the $U(j)$ action is the projection
$$ \Phi^{j}: \cc{O}_{\lambda} \rightarrow \lie{u}(j)^*$$
 sending every matrix to its $j\times j$
submatrix in top left corner. The moment map for the $T^{j}$ action 
$$\mu^{j}: \cc{O}_{\lambda} \rightarrow (\lie{t}^{j})^*$$
 sends the matrix $(a_{ij})$ 
to the diagonal $j \times j$ matrix $\diag\,(a_{11}, \ldots, a_{jj})$. In this way, we obtain additional Hamiltonian torus actions on $\cc{O}_{\lambda}$. However the dimension of torus acting effectively is much less then half of the dimension of the coadjoint orbit, so this action is still not toric. The Gelfand-Tsetlin pattern that we construct in the next section gives an action of an even bigger dimensional torus.\\
\indent
Now we analyze the moment map image $\cc{Q}=\mu(\cc{O}_{\lambda}) \subset (\lie{t}^n)^*$ for the standard $T^n$ action. The Vertices of $\cc{Q}$ correspond to the $T^n$-fixed points, that is, the diagonal matrices in $\cc{O}_{\lambda}$.
If $\lambda$ is generic, then the vertices correspond exactly to permutations on $n$ elements.
Thus there are exactly $n!$ of them.
If $\lambda$ is non-generic, say 
$$\lambda_1=\ldots =\lambda_{l_1} >\lambda_{l_1+1}=\ldots= \lambda_{l_1+l_2}> \ldots > \lambda_{n-l_s+1}= \ldots= \lambda_n ,$$
then the vertices correspond to cosets $S_n/( S_{l_1} \times \ldots \times S_{l_s})$, and there are
$\frac{n!}{l_1!\ldots l_s!}$ of them.
The coadjoint orbit $\cc{O}_{\lambda}$ with the standard $T^n$ action is a \textbf{GKM space}.
This means that the closure of every connected component of the set
$\{x \in \cc{O}_{\lambda};\, \dim (T^n \cdot x)=1\}$ is a sphere. The closure of $\{x \in \cc{O}_{\lambda};\, \dim (T^n \cdot x)=1\}$ is called \textbf{$1$-skeleton} of $\cc{O}_{\lambda}$. Denote by $\cc{Q}_1$ the image of $1$-skeleton under the moment map. Then $\cc{Q}_1$ is a graph with vertices $V(\cc{Q}_1)=V(\cc{Q})$ corresponding to $T^n$-fixed points and edges corresponding to closures 
of connected compontents of the $1$-skeleton. Note that not all edges in $\cc{Q}_1$ are edges of the polytope $\cc{Q}$. 
Images of two fixed points, $F$ and $F'$, are connected by an edge in $\cc{Q}_1$ if and only if they differ by one transposition of two different diagonal entries. Therefore there are exactly 
$$D:=[\,l_1(l_2 + \ldots l_s) + l_2 (l_3 + \ldots +l_s) + \ldots + l_{s-1}l_s\,]= \sum_{i<j}l_il_j$$
edges leaving any vertex of $\cc{Q}_1$ and thus $\dim\,\cc{O}_{\lambda}=D\,\dim(S^2)=2D$.
In the case of generic $\lambda$, the moment polytope of $\cc{O}_{\lambda}$ is called a permutahedron.
\\ \indent Denote the diagonal entries of $F$ by $F_{11}, \ldots ,F_{nn}$.
Let $p < q$ be indices from $\{1, \ldots n \}$ such that $F_{pp} \neq F_{qq}$
and $F'$ is the matrix obtained from $F$ by switching $p$-th and $q$-th entry.
The edge joining $\mu(F)$ and $\mu(F')$ is an $\mu$-image of a sphere in $\cc{O}_{\lambda}$ defined in the following way.
Denote $F_{pp}=v_i,\  F_{qq}=v_k$.
For any $z \in \bb{C}$ let $I_z$ be the matrix obtained from the identity matrix by changing four entries $(j,k)$ with $j,k \in \{p,q\}$ 
in the way presented below and let $F_z=I_z FI_z\inv$ be the matrix obtained from $F$ by conjugation with $I_z$. This means that $F_z$ differs from $F$ only at four entries $(j,k)$ with $j,k \in \{p,q\}$.
The matrices have the following shapes
\begin{displaymath}
I_z=\left[ \begin{array}{ccccc}
I &     \vdots  &  & \vdots& \\
   \ldots   & \frac{1}{Z}&  \ldots  & \frac{-\bar{z}}{Z} &  \ldots \\
 & \vdots      & I & \vdots& \\
   \ldots   & \frac{z}{Z} &  \ldots  & \frac{1}{Z} &  \ldots \\
 & \vdots      &  & \vdots& I\\
\end{array}
\right],\,\,\,
F_z=\left[ \begin{array}{ccccc}
\ddots & \vdots      & 0 & \vdots& 0\\
   \ldots   & \frac{(v_i +|z|^2 v_{k})}{Z}&  \ldots  & \frac{\bar{z}(v_i - v_{k})}{Z} &  \ldots \\
0 & \vdots      & \ddots & \vdots& 0\\
   \ldots   & \frac{z(v_i - v_{k})}{Z} &  \ldots  & \frac{(v_{k} +|z|^2 v_i)}{Z} &  \ldots \\
0 & \vdots      & 0 & \vdots& \ddots\\
\end{array}
\right]
\end{displaymath}
where $Z=\sqrt{1+|z|^2}$.
For more details about the moment image of standard torus action see for example \cite{Ty},\cite{MRS}.
%-------------------------------------------------------------------------------------------------

\subsection{Gelfand-Tsetlin system}\label{system}
In this subsection we recall the Gelfand-Tsetlin system of action coordinates, which originally appeared in \cite{GT}.
There are many references describing this system, for example \cite{GS}, \cite{K}, \cite{NNU}, \cite{H}.
For the readers' convenience and to fix the notation, we follow Mikhai Kogan's construction for a coadjoint $U(n)$ orbit in $\lie{u}(n)^*$, \cite{K}.
\\ \indent
Recall that the moment map for the $U(j)$ action on $\cc{O}_{\lambda}$, denoted $\Phi^{j}$,
maps $A \in \cc{O}_{\lambda}$ to $j \times j$ top left submatrix of $A$.
Denote the eigenvalues of $\Phi^{j}(A)$, ordered in a non-increasing way, by 
$$\lambda^{(j)}_1(A) \geq  \lambda^{(j)}_2(A) \geq \ldots \geq \lambda^{(j)}_{j}(A).$$
We will use the notation $\Lambda^{(j)}=(\lambda^{(j)}_1, \ldots, \lambda^{(j)}_{j}):\cc{O}_{\lambda} \rightarrow (\lie{t}^{j})^*_+ \hookrightarrow \bb{R}^j$, 
for a function sending $A$ to $(\lambda^{(j)}_1(A), \ldots  , \lambda^{(j)}_{j}(A))\in  \bb{R}^j$. Here we identify $(\lie{t}^{j})^*$ with $\bb{R}^j$ using pairings with positive roots. 
For $j=0$, we just get $\Phi^{n}(A)=A$ and $\lambda^{(n)}_j(A)=\lambda_j$.
The \textbf{Gelfand -Tsetlin system of action coordinates} is the collection of the functions $\lambda^{(j)}_{j}$ for $k=1, \ldots, n-1$ and $j=1, \ldots, k$.
We will denote them by $$\Lambda:\cc{O}_{\lambda} \rightarrow \bb{R}^N,$$ where
$$N:=(n-1)+(n-2)+ \ldots +1=\frac{n(n-1)}{2}.$$

Notice that $\Lambda^{(j)}$ is a composition of $\Phi^j$ and a map $s_j: \lie{u}(j)^* \rightarrow (\lie{t}^j)^*_+ \subset \bb{R}^j$ sending a point in $\lie{u}(j)^*$ to the unique point of intersection
of its $U(j)$ orbit with the positive Weyl chamber.
\begin{displaymath}
\xymatrix{
\mathcal{O}_{\lambda} \ar[r]^{\Phi^j} \ar[rd]_{\Lambda^{(j)}} &
\lie{u}(j)^*  \ar[d]^{s_j}&
 \\
&(\lie{t}(j))^*_+
} 
\end{displaymath}
Components of $s_j$ are $U(j)$ invariant, so they Poisson commute.
After precomposing them with $\Phi^j$, we get a family of Poisson commuting functions on $\cc{O}_{\lambda}$ (see Proposition 3.2 in \cite{GS}).
These are exactly $\lambda^{(j)}_1, \lambda^{(j)}_2, \ldots ,\lambda^{(j)}_{j}.$
For $l<j$ denote by $\kappa_{lj}:\lie{u}(j)^* \rightarrow \lie{u}(l)^* $ the transpose of the map $\lie{u}(l) \rightarrow \lie{u}(j) $
induced by the inclusion.
The functions $$\lambda^{(j)}_1, \lambda^{(j)}_2, \ldots ,\lambda^{(j)}_{j}, \lambda^{(l)}_1\circ\kappa_{lj}, \lambda^{(l)}_2\circ\kappa_{lj}, \ldots ,\lambda^{(l)}_{l}\circ\kappa_{lj}$$
Poisson commute on $\lie{u}(l)^* $ by Proposition 3.2 in \cite{GS} and the fact that first $j$ of them are $U(j)$ invariant.
Therefore all Gelfand-Tsetlin functions Poisson commute on $\cc{O}_{\lambda}$.

The classical mini max principle (see for example Chapter I.4 in \cite{CH}) implies that 
\begin{equation*}
 \lambda^{(l+1)}_j(A) \geq  \lambda^{(l)}_j(A) \geq \lambda^{(l+1)}_{j+1}(A).
\end{equation*}
We use the following notation for these inequalities:
\begin{equation}\label{ineq}
 \begin{array}{ccc}
 A_{l,j}:& & \lambda^{(l+1)}_j(A) \geq  \lambda^{(l)}_j(A),\\
B_{l,j}:& &  \lambda^{(l)}_j(A) \geq \lambda^{(l+1)}_{j+1}(A).
\end{array}
\end{equation}

The inequalities (\ref{ineq}) cut out a polytope in $\bb{R}^N$, which we denote by $\cc{P}$. 
%-----------------------------------------------------------------------------------------
\section{The action of the Gelfand-Tsetlin torus}\label{actionsection}
\subsection{Smoothness of the Gelfand-Tsetlin functions}\label{smooth}
The function $\lambda^{(j)}_k$ need not be smooth on the whole orbit $\cc{O}_{\lambda}$. 
The eigenvalues depend smoothly on the matrix entries, but this property is not preserved when reordering them in a non-increasing way.
They are smooth, however, on a dense open subset of $\cc{O}_{\lambda}$. To identify this subset we will need the following result proved in \cite{CDM}. 
This theorem is also true for orbifolds: see \cite[Theorem 3.1]{LMTW}.
\begin{theorem}\label{princ cross-section}
Let $G$ be a compact connected Lie group with a maximal torus $T$. 
Suppose $G$ acts on a compact connected symplectic manifold $M$ in a Hamiltonian way, with moment
map $\Phi: M \rightarrow \lie{g}^*$.
Then there exists a unique open wall $\sigma$ of the  Weyl chamber 
$ \lie{t}^*_+$ with 
the properties that 
$\Phi(M)\cap \lie{t}^*_+ \subset \bar{\sigma}$ and 
$\Phi(M)\cap \lie{t}^*_+\cap \sigma \neq \emptyset$. 
\end{theorem}
Let $\sigma=\sigma_j$ be the unique open wall from the above theorem applied to the standard $G=U(j)$ action on $M=\cc{O}_{\lambda}$. We call $\sigma$ the \textbf{principal face.} Any wall of positive Weyl chamber $ (\lie{t}^{j})_+^*$ that contains $\sigma$ is called a \textbf{ special wall}, while all the others walls are called \textbf{regular walls}.
Thus $\overline{\sigma}$ is the intersection of all special walls, and $\sigma =\overline{\sigma} \setminus (\cup \textrm{ regular walls})$.
Walls of $ (\lie{t}^{j})^*_+$ are defined by a collection of equations of the form 
$\lambda^{(j)}_{L}=\lambda^{(j)}_{L+1}$.
If a wall $\tau$ is special, i.e. $\overline{\sigma} \subset \tau$, then its defining equations hold on the whole $\Lambda(\cc{O}_{\lambda})$.
For any regular wall $\tau$, there is at least one of its defining equations, and some $A \in \cc{O}_{\lambda}$ such that $\Lambda(A)$ does not satisfy this equation.
\begin{proposition}
 The function $\Lambda^{(j)}$ is smooth on the set $U^{(j)}=(\Lambda^{(j)})\inv(\sigma)$.
\end{proposition}
\begin{proof}
To simplify the notation, we will denote $U(j)$ by $G$, and  the maximal torus in $U(j)$ simply by $T$.
Recall that the function $\Lambda^{(j)}$ is a composition of a smooth function $\Phi^{j}$ and projection 
$\pi:\lie{g}^*=\lie{u}(j)^* \rightarrow \lie{t}^*_+$. Therefore we only need to prove smoothness of the projection $\pi$ on $\Phi^{j}(U^{(j)})=\pi\inv(\sigma)$.
Note that all points in $\sigma$ have the same $G$-stabilizer (under the coadjoint action of $G$). Denote it by $H$.
Let $S$ be the subset of $\lie{g}^*$ equal to $\pi \inv (\sigma)$.  
This means that $S=(\lie{g}^*)_{(H)}$ is an orbit-type stratum and therefore it is a submanifold of $\lie{g}^*$.
Consider the smooth, $G$-equivariant, surjective map:
\begin{displaymath}
 \begin{array}{rcl}
 G \times \sigma & \rightarrow &S\\
(g,x)& \rightarrow & g\cdot x
\end{array}
\end{displaymath}
This map induces $G$-equivariant bijective map
\begin{displaymath}
 \begin{array}{rcl}
\Theta:G/H \times \sigma &\rightarrow S,\\
([g],x)& \rightarrow & g\cdot x
\end{array}
\end{displaymath}
which is also smooth (as $S$ is a manifold) and therefore it is a diffeomorphism (see for example Propositions 5.19 and 5.16 in \cite{Lee}).\\
Notice that the composition, $\pi \circ \Theta$
\begin{displaymath}
 \begin{array}{rcl}
 G/H \times \sigma & \rightarrow &\lie{t}^*_+\\
([g],x)& \rightarrow & x
\end{array}
\end{displaymath}
is just the projection onto second factor, therefore it is smooth.
This means that on $S$, $\pi$ is smooth, as a composition of $\Theta \inv$ and a smooth projection.
It follows that the function $\Lambda^{(j)}$ is smooth on the set $(\Phi^{j})\inv(S)=(\Lambda^{(j)})\inv(\sigma)=U^{(j)}$.
\end{proof}
\begin{remark}\label{moresmoothness}
The set of smooth points for $\Lambda^{(j)}$ may be strictly bigger than $U^{(j)}$.
For example, suppose that a function $\lambda^{(j)}_k$ is constant on the whole orbit $\cc{O}_{\lambda}$, 
and let $A$ be a point in $\cc{O}_{\lambda}$ such that $\lambda^{(j)}_k(A)=\lambda^{(j)}_{k+1}(A)$.
Suppose further that if for any $l \neq k$ we also have $\lambda^{(j)}_l(A)=\lambda^{(j)}_{l+1}(A)$
then $\lambda^{(j)}_l$ and $\lambda^{(j)}_{l+1}$ are equal on the whole $\cc{O}_{\lambda}$. 
In this case, the function $$\lambda^{(j)}_{k+1}=\textrm{trace} \circ \Phi^{j}-\sum_{l\neq k+1}\lambda^{(j)}_{l}$$ is smooth at the point $A$, 
as a difference of smooth functions, 
although $A$ is not in the set $U^{(j)}$ as defined above.
Proving the smoothness of the Gelfand-Tsetlin functions on a set bigger then $U^{(j)}$ would allow us to apply the proof of the main theorem to a wider class of non-generic coadjoint orbits.
The theorem holds if only there is a $T^n$-fixed point equipped with a smooth action of Gelfand-Tsetlin torus $T^D$.
Our techniques may be extended to coadjoint orbits with an additional eigenvalue repeating twice. The technical details became far more cumbersome, though, so we do not include them here.
\end{remark}

%-------------------------------------------------------------------------
%--------------------------------------------------------------------------

\subsection{The Torus action induced by the Gelfand-Tsetlin system.}\label{action}
At the points where $\Lambda^{(j)}$ is smooth, it induces a smooth action of $T^{j}$.
The process of obtainin this new action is often referred to as the \textbf{Thimm trick}.
An element $t \in T^{j}$ acts on a point $A \in \cc{O}_{\lambda}$ by the standard $U(j)$ action of $B^{-1}\, t\, B$, 
where $B \in U(j)$ is such that 
$B\,\Phi^{j}(A)\, B^{-1} \in (\lie{t}^{j})^*_+$. Denote this new action by $*$.
\begin{proposition}
The new $T^{j}$ action defined above is Hamiltonian on the subset $U^{(j)}=(\Lambda^{(j)})\inv(\sigma)$,
with moment map $\Lambda^{(j)}$.
\end{proposition}
%------------------------
\begin{proof}
To simplify the notation, we will denote $U^{(j)}$ simply by $U$.
Take any $X \in \lie{t}^{j}$ and denote by $X_{new}$ the vector field on $U$ generated by $X$ with $*$ action, 
and by $X_{std}$ the vector field on $U$ generated by $X$ using the standard action by conjugation.
As usual, for any function 
$\varphi: \cc{O}_{\lambda} \rightarrow \lie{u}(j)^*$, and any  $X \in \lie{u}(j)$, we denote by $\varphi^{X}$ a function from $\cc{O}_{\lambda}$ 
to $\bb{R}$ defined by $\varphi^{X}(p)=\langle \varphi(p), X \rangle$,
where $\langle, \rangle$ is the standard $U(j)$ invariant pairing between $\lie{u}(j)^*$ and $ \lie{u}(j)$.
Take any $A \in U$.
We want to prove that 
for any vector $Y \in T_A\cc{O}_{\lambda}=T_AU$
\begin{equation}\label{mmap}
 \omega(X_{new},Y)|_A=d\, (\Lambda^{(j)})^{X}\,(Y)|_A.
\end{equation}
Denote by $N$ the connected symplectic submanifold $N:=(\Phi^{j})^{-1}(\sigma) \subset \cc{O}_{\lambda}$,
where $\sigma$ is the principal face. We refer to $N$ as the \textbf{principal cross-section}. Note that $U=(\Lambda^{(j)})\inv(\sigma) =U(j) \cdot N$, and so every $A \in U$ 
can be $U(j)$ conjugated to an element of $N$.
We first prove equation (\ref{mmap}) for $A\in N$.\\
The proof of theorem 3.8 in \cite{LMTW} implies that
$$T_A \cc{O}_{\lambda}=T_A N \,+\, T_A(U(j) \cdot A).$$
This is not a direct sum. Thus to prove the equation (\ref{mmap}) for $A\in N$, 
it is enough to consider two cases: when vector $Y$ is tangent to the principal cross-section, and when it is tangent to $U(j)$ orbit (for the standard action).

Before we start considering the cases, we fix some notation.
For any vector field $V$ on $\cc{O}_{\lambda}$, denote by $\Psi^V$ its flow. 
Recall that $\Psi_{-t}^V=(\Psi_t^V)^{-1}$. Therefore, for example $\Psi_{t}^{X_{std}}(Q)=X_t QX_t^{-1}$ and 
$\Psi_{-t}^{X_{std}}(Q)=X_t^{-1} QX_t$.\\
%************************************************************************************************
\indent \textbf{Case 1}: Take $Y \in T_A N \subset T_A \cc{O}_{\lambda}.$
 We want to compute $\omega(X_{new},Y)|_A=\langle A, [X_{new}, Y] \rangle$.
Notice that on the principal cross section functions $\Phi^{j}$ and $\Lambda^{(j)}$ are equal, and the standard and the new actions of $T^{j}$ coincide.
Therefore the vector fields $X_{std}$ and $X_{new}$ have equal values and flows on $N$.
Using the formula
$$[X_{new},Y]=	\displaystyle\lim_{t\to 0}\frac{(\Psi^{X_{new}}_{-t})_*(Y)-Y}{t}=[X_{std},Y].$$
we have that, if $Y \in T_A N$, then $\langle A, [X_{new}, Y] \rangle=\langle A,[X_{std},Y] \rangle$.
The fact that functions $\Phi^{j}$ and $\Lambda^{(j)}$ agree on all of the $N$, means also that for $Y \in T_A N$
we have
$$ d(\Phi^{j})^X(Y)=d(\Lambda^{(j)})^X(Y).$$
Therefore \begin{eqnarray*}
\omega(X_{new},Y)|_A&=&\langle A, [X_{new}, Y] \rangle=\langle A, [X_{std}, Y] \rangle\\&=&\omega(X_{std},Y)|_A=d(\Phi^{j})^X(Y)|_A\\&=&d(\Lambda^{(j)})^X(Y)|_A.
          \end{eqnarray*}
%----------------------------------------------------------------------------------------------
\indent \textbf{Case 2}: Take $Y \in T_A(U(j) \cdot A)$. That is $Y=Y_{std}$ for some $Y=\frac{d}{dt}Y_t|_{t=0} \in \lie{u}(j)$ and the integral curve of $Y$ 
through $A$
is $\Psi^Y_t(A)= Y_t\, A \,Y_t^{-1}$.
As before, we start by analyzing $[X_{new},Y]$ at $A$. We have:
$$[X_{new},Y]|_A=\displaystyle\lim_{t\to 0}\frac{(\Psi^{X_{new}}_{-t})_*(Y)|_{\Psi^{X_{new}}_{t}(A)}-Y|_A}{t}.$$
The point $A$ is in $N$, so $\Psi^{X_{new}}_{t}(A)=X_t\cdot A= X_t\,A\, X_t^{-1}.$
Now we need to understand the expression:
$$(\Psi^{X_{new}}_{-t})_*(Y)|_{\Psi^{X_{new}}_{t}(A)}=\frac{d}{dv}\,\Psi^{X_{new}}_{-t}(Y_v\,\Psi^{X_{new}}_{t}(A)\,Y_v^{-1})\,|_{v=0}.$$
To compute the value of $\Psi^{X_{new}}_{-t}$ on $Y_v\,\Psi^{X_{new}}_{t}(A)\,Y_v^{-1}$, we need to find an element $C$ of $U(j)$ 
that would conjguate $\Phi^{j}(\Psi^{X_{new}}_{-t})$ to some element in $(\lie{t}^{j})^*_+$.
We have
\begin{displaymath}
 \begin{array}{cl}
 \Phi^{j}(Y_v\,\Psi^{X_{new}}_{t}(A)\,Y_v^{-1})&=\Phi^{j}(Y_v\,X_tA\, X_t^{-1}\,Y_v^{-1} )\\
 &=Y_v\,X_t\, \Phi^{j}(A)X_t^{-1}Y_v^{-1}.
 \end{array}
\end{displaymath}
Therefore, for 
$$C=X_t^{-1} \, Y_v^{-1}$$
we have that
$$C \Phi^{j}(Y_v\,\Psi^{X_{new}}_{t}(A)\,Y_v^{-1})\,C^{-1}=  \, \Phi^{j}(A)\, \in (\lie{t}^{j})^*_+ .$$
This means that the new action of $X_{t}$ at a point $Y_v\,\Psi^{X_{new}}_{t}(A)\,Y_v^{-1}$ is the same as standard action of 
$$C^{-1}\,X_{t}\,C= Y_v X_t\,X_{t}\,X_t^{-1} \,  Y_v^{-1}=Y_v  X_t\,  Y_v^{-1},$$
so 
\begin{displaymath}
 \begin{array}{cl}
  & \Psi^{X_{new}}_{-t}(Y_v\,\Psi^{X_{new}}_{t}(A)\,Y_v^{-1})\\
&=(Y_v X_t^{-1} \,  Y_v^{-1})(\,Y_v\, X_t \,A\, X_t^{-1} \,Y_v^{-1})( Y_v   X_t    Y_v^{-1})\\  
&=Y_v\,A\, Y_v^{-1}.
 \end{array}
\end{displaymath}  
Therefore
$$[X_{new},Y]|_A=\displaystyle\lim_{t\to 0}\frac{(\Psi^{X_{new}}_{-t})_*(Y)|_{\Psi^{X_{new}}_{t}(A)}-Y|_A}{t}=
\displaystyle\lim_{t\to 0}\frac{Y|_A-Y|_A}{t}=0,$$
and
$$\omega(X_{new},Y)|_A = \langle A, [X_{new}, Y] \rangle=0.$$
Notice that the function $\Lambda^{(j)}$ is constant on $U(j)$ orbits, because $\Phi^{j}$ is $U(j)$-equivariant and 
the whole $U(j)$ orbit intersects $(\lie{t}^{j})^*_+$ in a unique point. Thus, for $Y \in T_A(U(j)~\cdot~A)$,
$$d\, (\Lambda^{(j)})^{X}\,(Y)=0.$$
and equation (\ref{mmap}) for $A$ in $N$ follows.
\\ \indent
Now we want to prove equation (\ref{mmap}) for all $C \in U$. Let $B$ be an element of $U(j)$ such that $B C B\inv =A \in \lie{t}^*_+$.
Take any $X \in \lie{t}$ and $Y \in T_C U$.
Using the $U(j)$ invariance of $\omega$ and of $\Lambda^{(j)}$, and equation (\ref{mmap}) at the principal cross section, we have
\\
\begin{align*}
\omega ( X_{new},Y)_{|B\inv A B}&=\omega(\,\ \frac{d}{dt}(B\inv X_t B \cdot C)|_{t=0},\, \frac{d}{dt}(\Psi^Y_t(C))|_{t=0}\, \,) \\
&=\omega(\,\ \frac{d}{dt}B(B\inv X_t B \cdot C)B\inv|_{t=0},\, \frac{d}{dt}B(\Psi^Y_t(C))B\inv|_{t=0})\\
&=\omega(\,\ \frac{d}{dt}( X_t B B\inv A B B\inv X_t\inv)|_{t=0},\, \frac{d}{dt}(\Psi^{BYB\inv}_t(A))|_{t=0})\\
&=\omega(X_{new}, BYB\inv)|_{A}=d\,(\Lambda^{(j)})^X (BYB\inv)|_A\\
&=\frac{d}{dt}\,{[}\,(\Lambda^{(j)})^X(B\Psi^Y_t(C)B\inv))\,{]}|_{t=0}=\frac{d}{dt}\,{[}\,(\Lambda^{(j)})^X(\Psi^Y_t(C))\,{]}|_{t=0}\\
&=d\,(\Lambda^{(j)})^X (Y)|_C,
\end{align*} which is exactly what we needed to show.
\end{proof}
Putting this together for all $k$ gives us a Hamiltonian (although not necessarily effective) action of $T^N$ on the open dense subset, 
$$U:=\bigcap_jU^{(j)}.$$
We call a wall of $(\lie{t}^N)^*_+$ \textbf{special} if there is a $j$ such that the image of this wall under projection
$(\lie{t}^N)^* \rightarrow (\lie{t}^{j})^*$ is a special wall as defined in the Section \ref{smooth}. Other walls of $(\lie{t}^N)^*_+$
will be called \textbf{regular}.\\ \indent
Notice that the standard action of $T^n$, described in the Section \ref{standard action}, is a part of the $T^N$ action on $U$.
One can easily compute the $T^n$-moment map $\mu$, which mapps a matrix to its diagonal entries, from $\Lambda$.
Of course $\lambda^{(1)}_1(A)=a_{11}$. Using the fact that the trace of $\Phi^{2}(A)$ is $a_{11}+ a_{22}=\lambda^{(2)}_1(A)+\lambda^{(2)}_2(A)$
we compute the value $a_{22}$. Continuing this process we obtain all the diagonal entries of $A$, that is we obtain $\mu(A)$.
This defines the projection $pr: (\lie{t}^N)^* \rightarrow(\lie{t}^n)^*$, which on 
the image of $\Lambda$ is given by the following formula
$$ pr (\{\lambda^{(j)}_l\})=\Bigl( \lambda^{(1)}_1,\,(\lambda^{(2)}_1+\lambda^{(2)}_2-\lambda^{(1)}_1)\,,
 \ldots ,\, \sum_i \lambda^{(n-1)}_i\, -\sum_i \lambda^{(n-2)}_i, \sum_i \lambda^{(n)}_i\, -\sum_i \lambda^{(n-1)}_i \Bigr).$$
This means $\mu = pr \circ \Lambda $. 
Under this projection, the Gelfand-Tsetlin polytope $\cc{P}$, described below, maps to the moment map image, $\cc{Q}$, of the standard maximal torus action.
Here is an example for a generic $SU(3)$ coadjoint orbit, $\cc{O}_{\lambda}$.
$$
\begin{array}{cc}
 \includegraphics[width=.2\textwidth]{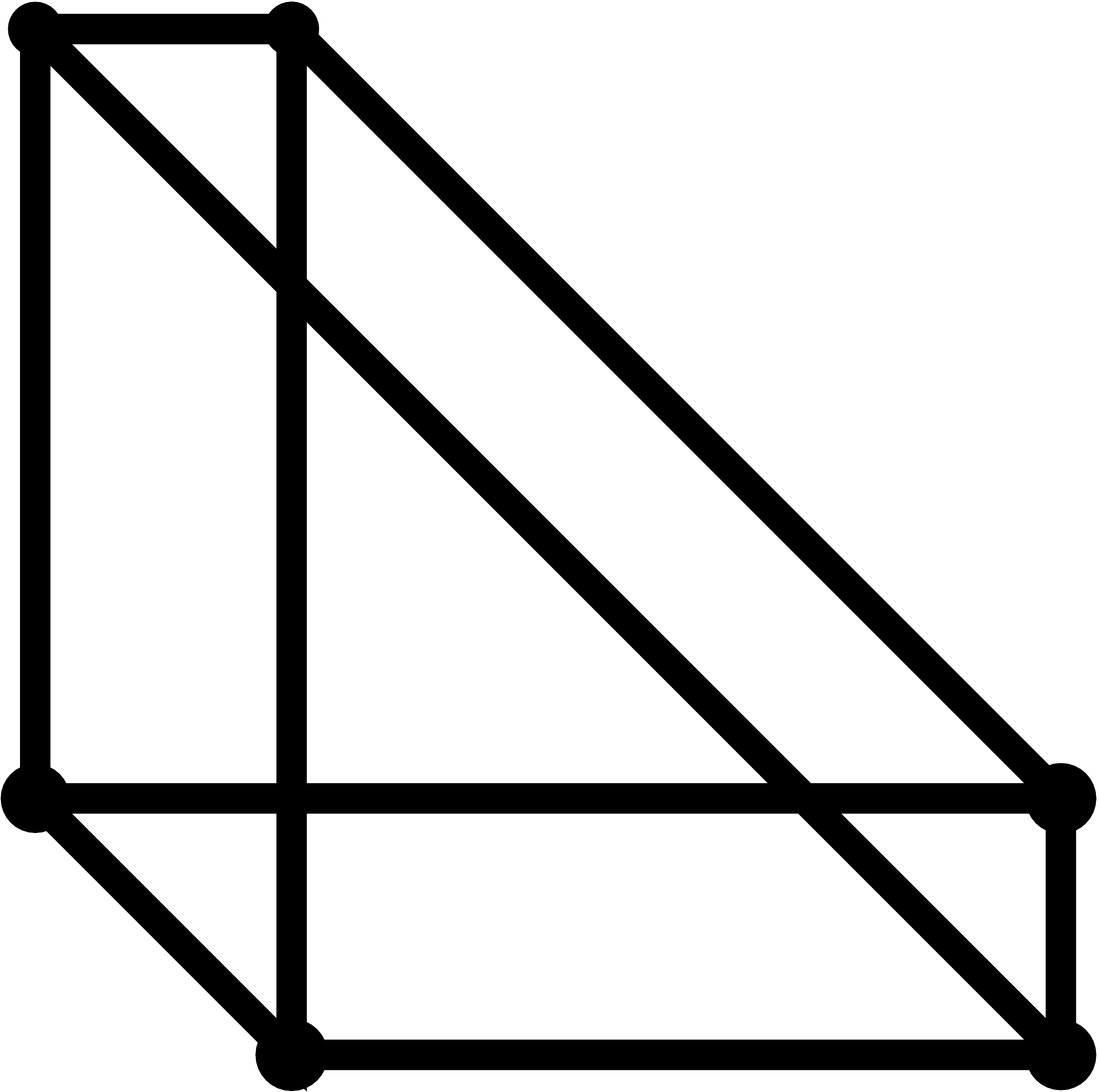}\,\, & \,\,
 \includegraphics[width=.2\textwidth]{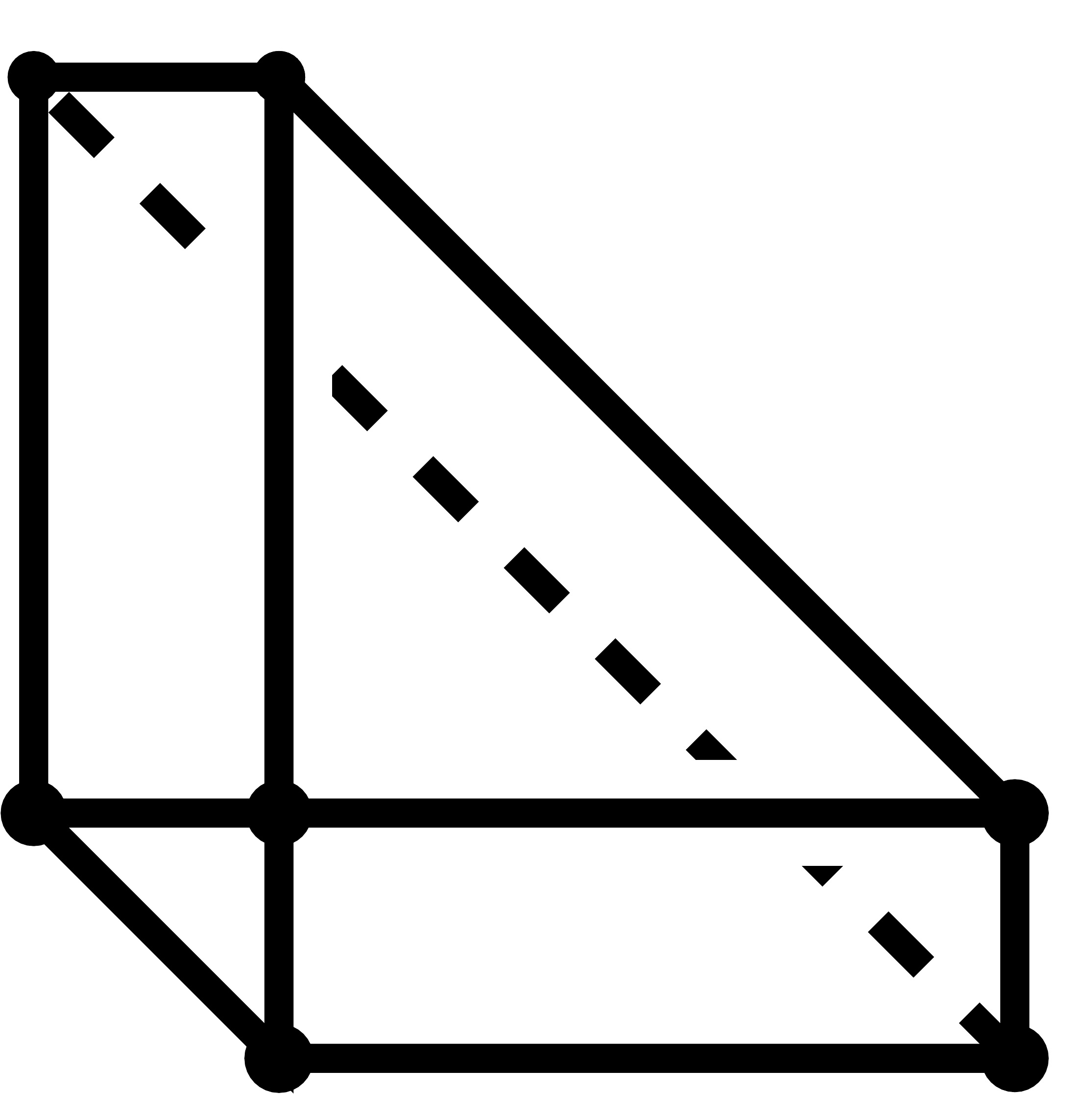}
\\
& \\
\cc{Q}=\mu(\cc{O}_{\lambda}) \in \bb{R}^2 & \cc{P}=\Lambda(\cc{O}_{\lambda}) \in \bb{R}^3
\end{array}
$$
%-------------------------------------------------------------------------------------------------
%------------------------------------------------------------------
\subsection{The Gelfand-Tsetlin polytope}
In this subsection we analyze the image $\Lambda(\cc{O}_{\lambda})$ in $\bb{R}^N$, where $N:=n(n-1)/2$.
The inequalities (\ref{ineq}) cut out a polytope in $\bb{R}^N$, which we denoted by $\cc{P}$, and  $\Lambda(\cc{O}_{\lambda})$ is contained in this polytope.
\begin{proposition}\label{gtpolytope}
 The image $\Lambda(\cc{O}_{\lambda})$ is exactly $\cc{P}$.
\end{proposition}
\begin{proof}
The Proposition follows from successive applications of the following lemma (Lemma 3.5 in \cite{NNU}, see also \cite{GS2}), as explained below.
\begin{lemma}\label{open}
 For any real numbers $a_1 \geq b_1 \geq a_2 \geq \ldots \geq a_k \geq b_k \geq a_{k+1}$ there exist $x_1, \ldots, x_k$ in $\bb{C}$ and $x_{k+1}$ in $\bb{R}$ such that
the Hermitian matrix  
\begin{displaymath}
A:=\left(\begin{array}{cccc}
    b_1&  & 0& \bar{x}_1\\
& \ddots & & \vdots\\
0& & b_k & \bar{x}_k\\
x_1& \hdots & x_k& x_{k+1}
    \end{array}
\right),
\end{displaymath}
has eigenvalues $a_1, \ldots , a_{k+_1}$.
\end{lemma}
Now let $c_1, \ldots, c_{k-1}$ be numbers such that
$b_1 \geq c_1 \geq b_2 \ldots \geq b_{k-1} \geq c_{k-1} \geq b_k.$
Applying Lemma \ref{open} again, we get that there exist $y_1, \ldots , y_{k-1}$ in $\bb{C}$ and $y_{k}$ in $\bb{R}$ such that
the Hermitian matrix
\begin{displaymath}
B:=\left(\begin{array}{cccc}
    c_1&  & 0& \bar{y}_1\\
& \ddots & & \vdots\\
0& & c_{k-1} & \bar{y}_{k-1}\\
y_1& \hdots & y_{k-1}& y_k
    \end{array}
\right),
\end{displaymath}
has eigenvalues $b_1, \ldots , b_{k}$.
Therefore there is an invertible matrix $C \in U(k)$ such that $CBC\inv=\diag (b_1, \ldots, b_k)$.
Denote by $X$ the column vector $(x_1, \ldots, x_k)^T$. Notice that
\scriptsize 
\begin{displaymath}
\left(\begin{array}{ccc|c}
    &  & & 0\\
& C & & \vdots\\
& &   & 0\\
\hline
0& \hdots & 0& 1
    \end{array}
\right)
\left(\begin{array}{ccc|c}
    &  & & \\
& B & & C\inv \overline{X}\\
 &  & \\
\hline
&X^T C & & x_{k+1}
    \end{array}
\right)
\left(\begin{array}{ccc|c}
    &  & & 0\\
& C\inv & & \vdots\\
& &   & 0\\
\hline
0& \hdots & 0& 1
    \end{array}
\right)=
\left(\begin{array}{ccc|c}
    &  & & \\
& CBC\inv & & C\,C\inv \overline{X}\\
 &  & \\
\hline
&X^T C\,C \inv & & x_{k+1}
    \end{array}
\right)=A
\end{displaymath}
\normalsize
Therefore the Hermitian matrix 
\begin{displaymath}
\left(\begin{array}{ccc|c}
    &  & & \\
& B & & C\inv \overline{X}\\
 &  & \\
\hline
&X^T C & & x_{k+1}
    \end{array}
\right)
\end{displaymath}
has desired values of the Gelfand-Tsetlin functions $\lambda^{(k+1)}_*,\lambda^{(k)}_*, \lambda^{(k-1)}_*$.
Continuing this process, we construct a matrix $A$ in $\cc{O}_{\lambda}$
such that $\Lambda(A)=L$, for any chosen point $L$ in the polytope $\cc{P}$.
\end{proof}
The polytope $\cc{P}\subset \bb{R}^N$ is called the \textbf{Gelfand-Tsetlin polytope}.
We think of $\bb{R}^N$ as having coordinates $\{x^{(j)}_{k}\}$, indexed by pairs $(j,k)$, for $j=1,\ldots,n-1$, and $k=1, \ldots, j$,
so that $x^{(j)}_{k}$-th coordinate of $\Lambda(A)$ is $\lambda^{(j)}_{k}(A)$.
\begin{lemma}\label{vertex}
Let $\Lambda=\Lambda(A)$, $A \in \cc{O}_{\lambda}$,
be a point in the polytope $\cc{P}$,
with coordinates $\{\lambda^{(j)}_{k}(A)\}$.
Suppose that for any $(j,k)$, $j=1,\ldots,n-1$, $k=1, \ldots, j$, we have that
$$ \lambda^{(j)}_{k}(A)=\lambda^{(j+1)}_{k}(A)\textrm{   or   }\lambda^{(j)}_{k}(A)=\lambda^{(j+1)}_{k+1}(A).$$
Then $\Lambda$ is a vertex of the polytope $\cc{P}$.
\end{lemma}
\begin{proof}
For any pair $(j,k)$ pick one equality, $A_{j,k}$ or $ B_{j,k}$, that is satisfied by $\Lambda$ (if both are satisfied pick either one of them).
Arrange these inequalities to be of the form: 
\begin{center}
(linear combination of variables $x^{(j)}_{k}$) $\leq$ real constant.
\end{center}
Sum all of these $N$ inequalities together, forming the inequality 
$$CX \leq Z,$$
where $X=(x^{(n-1)}_{1}, \ldots, x^{(1)}_{1}) \in \bb{R}^N$ is the variable, and $Z,C \in \bb{R}^N$ are constants.
Every $X \in \cc{P}$ has to satisfy $CX \leq Z$, as this is just a sum of $N$ of the $2N$ inequalties defining $\cc{P}$.
Therefore $\cc{P} \cap \{X;\, CX=Z\}$ is a face of $\cc{P}$,
(see Definition 2.1 in \cite{Zi}).
Note that $X \in \cc{P}$ satisfies $CX=Z$ if and only if all of the $N$ inequalities defining $\cc{P}$ we have summed, are equalities for $X$.
This determines the values of all $x^{(j)}_{k}$ in terms of $\lambda_1, \ldots, \lambda_n$.
Therefore 
$$\cc{P} \cap \{X;\, CX=Z\}=\{\Lambda(A)\}$$
is a $0$-dimensional face, in other words a vertex of $\cc{P}$.
\end{proof}
%----------------------------------------------------
\noindent To emphasize the main idea of this proof, we give the following example.
\begin{example}
 Let $n=3$, $\lambda=(5,5,4)$ and $\Lambda(A)=(\lambda^{(2)}_{1}(A),\lambda^{(2)}_{2}(A),\lambda^{(1)}_{1}(A))=(5,4,5)$.
We need to choose inequalities $A_{j,k}, B_{j,k}$, one for each pair $(j,k)$, that are equalities for $\Lambda(A)$.
For $\lambda^{(2)}_{1}(A)$ we have a choice as both of them are equations. Say we pick $B_{2,1}$, $B_{2,2}$ and $A_{1,1}$.
The set of rearranged inequalities is
 \begin{eqnarray*}
  -x^{(2)}_{1}&\leq& -\lambda_2=-5\\
-x^{(2)}_{2}&\leq& -\lambda_3=-4\\
x^{(1)}_{1}-x^{(2)}_{1}& \leq& 0
 \end{eqnarray*}
Summing these inequalities together we obtain
$$ -2x^{(2)}_{1}-x^{(2)}_{2}+x^{(1)}_{1}\leq -9.$$
This inequality is satisfied on all $\cc{P}$. 
An element $X \in \cc{P}$ satisfies $ -2x^{(2)}_{1}-x^{(2)}_{2}+x^{(1)}_{1}=-9$ if and only if
\begin{eqnarray*}
  -x^{(2)}_{1}&=&-5\\
-x^{(2)}_{2}&=&-4\\
x^{(1)}_{1}&=&x^{(2)}_{1}.
 \end{eqnarray*}
Thus, we see that $(5,4,5)$ is the unique solution to these inequalities in $\cc{P}$.
\end{example}
%------------------
\begin{lemma}\label{fixedptvertex} The map $\Lambda$ sends every $T^n$=fixed point to a vertex of $\cc{P}$.
\end{lemma}
\begin{proof}
 For a diagonal matrix $F=\diag(F_{1,1}, \ldots, F_{n,n})$, the set of eigenvalues of $F_{j+1}:=\Phi^{j+1}(F)$ is obtained from 
the set of eigenvalues of $F_{j}:=\Phi^{j}(F)$ by adding $F_{j+1,j+1}$. Let $s$ be such that
$$ \lambda^{(j)}_{s}(F)\geq F_{j+1,j+1}>\lambda^{(j)}_{s+1}(F).$$
Then
\begin{eqnarray*}
  \forall_{l\leq s} & \lambda^{(j)}_{l}(F) & =\lambda^{(j+1)}_{l}(F)\\
 \forall_{l > s} & \lambda^{(j)}_{l}(F) & =\lambda^{(j+1)}_{l+1}(F).
 \end{eqnarray*}
Therefore $\Lambda(F)$ is a vertex of $\cc{P}$, by Lemma \ref{vertex}.
\end{proof}

%---------------------------------------
\begin{lemma}\label{edge}
Let $\Lambda=\Lambda(A)$, for $A \in \cc{O}_{\lambda}$,
be a point in the polytope $\cc{P}$,
with coordinates $\{\lambda^{(j)}_{k}(A)\}$.
Suppose that 
there exists exactly one pair of indices $(j_0,k_0)$ such that both inequalities $A_{j_0,k_0}$ and $ B_{j_0,k_0}$ at the point $A$ are strict.
That is, for all $(j,k)\neq (j_0.k_0)$, $j=1,\ldots,n-1$, $k=1, \ldots, j$, we have one of the equalities
$$ \lambda^{(j)}_{k}(A)=\lambda^{(j+1)}_{k}(A)\textrm{   or   }\lambda^{(j)}_{k}(A)=\lambda^{(j+1)}_{k+1}(A).$$
Then $\Lambda(A)$ is contained in the interior of an edge of $\cc{P}$.
\end{lemma}
\begin{proof}
 Proceed similarly as in the proof of Lemma \ref{vertex}.
For any $(j,k)\neq (j_0.k_0)$ choose one of the inequalities $A_{j,k}, B_{j,k}$ that is equality for $\Lambda(A)$.
Arrange these inequalities to be of the form: 
\begin{center}
(linear combination of variables $x^{(j)}_{k}$) $\leq$ real constant.
\end{center}
Sum all of these $N-1$ inequalities together forming the inequality 
$$CX \leq Z.$$
As before, this gives an inequality valid for $\cc{P}$, and 
$\cc{P} \cap \{X;\, CX=Z\}$ is a face of $\cc{P}$.
The equation $CX=Z$ determines the values of all
$x^{(j)}_{k}$, with $(j,k)\neq (j_0,k_0)$, in terms of $\lambda_1, \ldots, \lambda_n$ and $x^{(j_0)}_{k_0}$.
These uniquely determined values are $x^{(j)}_{k}=\lambda^{(j)}_{k}(A)$.
For any assignement of the value for $x^{(j_0)}_{k_0}$, the equation $CX=Z$ will still hold.
In order to have $X\in \cc{P}$ we need to pick the value for $x^{(j_0)}_{k_0}$
in the open interval $(x^{(j_0+1)}_{k_0},x^{(j_0+1)}_{k_0+1})=(\lambda^{(j_0+1)}_{k_0}(A),\lambda^{(j_0+1)}_{k_0+1}(A))$.
Note that $\lambda^{(j_0+1)}_{k_0}(A)\neq \lambda^{(j_0+1)}_{k_0+1}(A)$ because if they were equal, then they would also be equal to
$\lambda^{(j_0)}_{k_0}(A)$ what contradicts our assumptions. Thus we really are choosing the value for 
$x^{(j_0)}_{k_0}$
from the open, non-degenerate interval $(\lambda^{(j_0+1)}_{k_0}(A),\lambda^{(j_0+1)}_{k_0+1}(A))$.
Therefore 
$$\cc{P} \cap \{X;\, CX=Z\}\cong (\lambda^{(j_0+1)}_{k_0}(A),\lambda^{(j_0+1)}_{k_0+1}(A))$$
is a $1$-dimensional face of $\cc{P}$.
\end{proof}
\begin{proposition}
 For any $\lambda$, the dimension of the polytope $\cc{P}$ is half of the dimension of $\cc{O}_{\lambda}$.
\end{proposition}
\begin{proof}
Fix $\lambda \in (\lie{t}^n)^*_+$, not necessarily generic.
Let $l_1, \ldots, l_s$ be the integers such that $l_1+ \ldots +l_s=n$ and
$$\lambda_1=\ldots =\lambda_{l_1} >\lambda_{l_1+1}=\ldots= \lambda_{l_1+l_2}> \ldots > \lambda_{n-l_s+1}= \ldots= \lambda_n .$$
Consider the coadjoint orbit $M:=\mathcal{O}_{\lambda}$ in $U(n)$.
The dimension of $\mathcal{O}_{\lambda}$ was already computed in Section \ref{standard action} and is equal to
$$2D:=2\,[\,l_1(l_2 + \ldots l_s) + l_2 (l_3 + \ldots +l_s) + \ldots + l_{s-1}l_s\,]= 2\,\sum_{i<j}l_il_j.$$
If some $l_j > 1$, then the $(l_j-1)$ functions $\lambda^{(1)}_{l_1+\ldots+l_{j-1}+1}=\ldots =\lambda^{(1)}_{l_1+\ldots+l_{j}-1}$
have to be equal to $\lambda_{l_1+\ldots+l_{j-1}+1}$ due to inequalities (\ref{ineq}).
Lemma \ref{open} implies that the image $\Lambda^{(1)}(\cc{O}_{\lambda})$ in $(\lie{t}^{n-1})^*\cong \bb{R}^{n-1}$
has dimension equal to the number of non-constant functions from $\lambda^{(1)}_{*}$
that is $$n-1-\sum_{j=1}^s(l_j-1).$$
Inequalities (\ref{ineq})
force also $(l_j-2)$ of functions $\lambda^{(2)}_*$ to be equal to $\lambda_{l_1+\ldots+l_{j-1}+1}$, as well as $l_j-3$ of functions $\lambda^{(3)}_*$, etc.
The number of our functions $\lambda^*_*$ that are constant is
$$\frac{l_1(l_1-1)}{2}+ \ldots +\frac{l_s(l_s-1)}{2}.$$
The remaining functions form the system of action coordinates, consisting of
$$ \frac{n(n-1)}{2} -\left( \frac{l_1(l_1-1)}{2}+ \ldots +\frac{l_s(l_s-1)}{2} \right) =\sum_{i<j}l_il_j=D$$
independent functions (see Proposition \ref{gtpolytope} and its proof). Therefore
the dimension of the image $\Lambda(\cc{O}_{\lambda})$ is $D$.
\end{proof}
For non-generic orbits, have $D \neq N$ and $T^N$ action is not effective. 
Let $\bb{R}^D$ be the smallest subspace of $\bb{R}^N\cong (\lie{t}^N)^*$ containing the 
polytope $\cc{P}$,
and let $T^D \hookrightarrow T^N$ be the corresponding subtorus of $T^N$.
Then the action of $T^D$ is effective and Hamiltonian on $U=\bigcap_j U^{(j)}=\bigcap_j \,(\Lambda^{(j)})\inv(\sigma_j)$.\\
\indent If $\cc{F}$ is a face of $\cc{P}$ containing some $x\in \Lambda(U)$, then, by the definition of $U$, $x$ is not on any regular wall.
Therefore any point of the interior $\cc{F}$ also cannot be on any regular wall, so it is in $U$.
\begin{lemma}\label{fixednotreg}
 If $\lambda$ is generic, then the images of fixed points of standard $T^n$ action are in $U$.
If $\lambda$ is non generic but there is only one eigenvalue that is repeated - then there is a $T^n$-fixed point that is in $U$.
\end{lemma}
\begin{proof}
If $\lambda$ is generic, then for any $T^n$-fixed point $F$ and any $k$, the matrix $\Phi^{j}(F)$ is a diagonal matrix with all diagonal entries distinct. Therefore 
$\Lambda(F)$ is not on any regular wall, so it is in $U$.\\
Now assume that $\lambda$ is of the form
$$\lambda_1 > \lambda_2> \ldots > \lambda_{l_1}=\lambda_{l_1+1}= \ldots = \lambda_{l_1+s}>\lambda_{l_1+s+1}> \ldots > \lambda_n.$$
Let $\{v_1>v_2>\ldots >v_{n-s}\}=\{\lambda_1 >\lambda_2> \ldots > \lambda_l>\lambda_{l_1+s+1}> \ldots >\lambda_n\}$
be the set of distinct eigenvalues.
Consider the $T^n$-fixed point
\begin{displaymath}
 F=\left(\begin{array}{c|c}
    A& 0\\
\hline
0&\lambda_{l_1} \textrm{Id}_{s}
    \end{array}
\right)
\end{displaymath}
where $A$ is any diagonal $(n-s) \times (n-s)$ matrix with spectrum $\{v_1,v_2,\ldots ,v_{n-s}\}.$
The figure below presents the values of Gelfand-Tsetlin functions $\lambda^{(j)}_{k}$ at $F$, for $j \geq n-s$ 
For $j \leq n- s$ the values $\lambda^{(j)}_{1}(F), \ldots , \lambda^{(j)}_{j}(F)$ are all distinct.
\scriptsize
\begin{displaymath}
 \begin{array}{cccccccccccccccccccccccc}
  v_1 &     &\ldots &  & v_{l_1-1} &  & v_l  &  &  & \ldots &  &  &  v_{l_1} &  &   v_{l_1+1} &  & \ldots &  &   v_{n-s} &\\

× & ×v_1 &       &\ldots &  & v_{l_1-1} &  & v_l  &  & \ldots  &   &   v_{l_1} &  &   v_{l_1+1} &  & \ldots &  &   v_{n-s} &  &\\ 

   &  &   \ddots &  &    \ddots  & &   &    & &  \vdots  &  &  & &  &  \begin{rotate}{70} $\ddots$ \end{rotate} & &&  &  & \\

 × & × &  &  v_{l_1} &       &\ldots  &  & v_{l_1-1} &  &v_{l_1} &  & v_{l_1+1} &    &\ldots&   &   v_{n-s}  &    & \\

\end{array}
\end{displaymath}
\normalsize
\\
Therefore $\lambda^{(k)}_{j}=\lambda^{(k)}_{j+1}$ at $F$ if and only if this equation is valid for the whole orbit.
This shows that the fixed point $F$ of the form described above is in the set $U$.
\end{proof}
%-------------------------------------------------------
We call $\Lambda$ images of such $T^n$-fixed points \textbf{good vertices} of $\cc{P}$. \\ \indent
Consider for example the non-generic $\lambda=(5,4,4,4,3,1).$
Here is the $T^n$-fixed point and its Gelfand-Tsetlin functions (the bold ones are constant on the whole orbit)
\small
\begin{displaymath}
 F=\left(\begin{array}{ccc|ccc}
  1&&  &&&\\
&5&&&&\\
&&3  &&&\\

\hline
&&& 4&&  \\
&&&&4&\\
&&&&&4  \\
    \end{array}
\right),\,\,\,
\begin{array}{ccccccccc}
  5&&\textbf{4}&&\textbf{4}&&3&&1\\
&5&&\textbf{4}&&3&&1&\\
&&5&&3&&1  &&\\
&&&5&&1&&&\\
&&&&1&&&&
    \end{array}\\
\end{displaymath}
\normalsize
Take any good vertex $V_F=\Lambda(F)$.
\begin{proposition}\label{Dedges}
There are exactly $D$ edges in $\cc{P}$ emanating from $\Lambda(F)$.
\end{proposition}
\begin{proof}
All the $\Lambda$ preimages of interiors of faces containing $\Lambda(F)$, are also in $U$.
Thus around $F$ we have a smooth, effective, Hamiltonian action of $T^D$ on $U$.
The local normal form theorem, (see for example \cite{KT2}), gives that, in a suitably chosen basis,the image of moment map is a $D$ dimensional orthant.
In particular this proves that there are exactly $D$ edges starting from this point.
\end{proof}
Note that there may be more then $D$ edges starting from vertices of $\cc{P}$ that are not good vertices.
%----------------------------------------------------
\section{Proof of the main theorem}\label{mainproof}
\begin{proof} Recall that the main theorem states that the Gromov width of the coadjoint $U(n)$-orbit, $\mathcal{O}_{\lambda}$,
through $\lambda$ of the form
$$\lambda_1 > \lambda_2> \ldots > \lambda_{l_1}=\lambda_{l_1+1}= \ldots = \lambda_{l_1+s}>\lambda_{l_1+s+1}> \ldots > \lambda_n, \,\,s \geq 0,$$
is at least $\min\{\lambda_i-\lambda_j\,|\,\lambda_i > \lambda_j\}$. Let $$\{v_1>v_2>\ldots >v_{n-s}\}=
\{\lambda_1 >\lambda_2> \ldots > \lambda_{l_1}>\lambda_{l_1+s+1}> \ldots >\lambda_n\}$$
be the set of distinct eigenvalues.
Our main theorem states that in this case for any $r< \min\{v_i-v_{i+1}\}$
we can symplectically embed a ball $B^{2D}_r$ of capaciy $r$ into $\mathcal{O}_{\lambda}$.\\
\indent Let $\Lambda(F)=V_F$ be a good vertex of $\cc{P}$.
Let $\cc{T}$ be an open subset of $\lie{t}^*$ such that 
$$\Lambda(\mathcal{O}_{\lambda})\cap\cc{T}=\bigcup_{\substack{\cc{F} \text{ face of } \cc{P} \\ V_F \in \cc{F}}}(\text{rel-int } \cc{F})$$
and let
$\cc{W}=\Phi\inv(\cc{T})$.
This is the largest subset of $M$ centered around a point $F=\Lambda^{-1}(V_F)$ (compare with Example \ref{largest}).
Then $\cc{W}$ is centered around this vertex, according to Definition \ref{centered-definition}.
Proposition \ref{embed} gives us an symplectic embedding
$$\Psi:\left\{ z \in \bb{C}^D \ | \ V_F + \pi \sum |z_j|^2 \eta_j \in \cc{T} \right\}\rightarrow \cc{O}_{\lambda},$$
where $\eta_1, \ldots, \eta_D$ are the isotropy weights of $T^D$ action on $T_{F}\cc{O}_{\lambda}$.
These $D$ weights span $D$ edges of $\cc{P}$ starting from $V_F$. For the edge in the direction of $\eta_l$, there is a number $c_l \in \bb{R}$
such that the edge is precisely $c_l\,\eta_l$. Let 
$$r=\max\{s\,|\,s \leq c_l,\,\textrm{ for all }l=1,\ldots, D\}.$$
The ball of capacity $r$, $B_r=\{z \in \bb{C}^D\,| \pi \sum |z_l|^2\eta_l <r \}$,
is contained in the domain of $\Psi$. Therefore the restriction of $\Psi$ gives us a symplectic embedding of a ball of capacity $r$.\\
\indent We prove the main theorem by showing that for any edge, $c_l$ is at least 
the minimum $\min\{v_i-v_j\,|\,v_i>v_j\}=\min\{\lambda_i-\lambda_j\,|\,
\lambda_i>\lambda_j\}.$
Moreover, we will show that
any good vertex 
 has an edge with the length equal to the minimum of $v_i-v_j$ times the length of $\eta_l$ spanning this edge.
This means that the lower bound we prove is the best possible we can get from this almost toric action.
Let us emphasize that there might exist symplectic embeddings of bigger balls, however this method fails to find them.
%--------------------------------------------------------------------------------------------------------------------------------------------------
\begin{proposition}\label{edgelengths}
The length of any edge in $\cc{P}$ starting from $V_F$ is at least $\min\{v_i-v_j\,|\,v_i>v_j\}$ times the length of the weight spanning this edge.
Moreover, there is an edge with length exactly the $\min\{v_i-v_j\,|\,v_i>v_j\}$ times the length of the weight spanning it.
\end{proposition}
\begin{proof}
 Recall from Section \ref{action} that the moment maps for the standard and the Gelfand-Tsetlin torus actions are related through projection $pr$:
$\mu=pr \circ \Lambda$.
We will show that 
for any edge $e \in \cc{P}$ starting from $V_F$ there is an edge $e'$ in $\cc{Q}_1$ (possibly not and edge but just a line segment in $\cc{Q}$) 
such that $pr (e)\subset e'$.
This will help us to analyze edges of $\cc{P}$.\\
\indent Denote the diagonal entries of $F$ by $F_{11}, \ldots ,F_{nn}$.
Let $p < q$ be indices from $\{1, \ldots ,n \}$ such that $F_{pp} \neq F_{qq}$
and $F'$ is the matrix obtained from $F$ by switching $p$-th and $q$-th entry.
The edge joining $\mu(F)$ and $\mu(F')$ is an $\mu$-image of a sphere $S:=\{F_z;\, z \in \bb{C} \cup \{ \infty \}\,\}$ in $\cc{O}_{\lambda}$ defined in the Section \ref{standard action}.
We will analyze $\Lambda(S)$.\\
Assume that $v_k < v_i$. The other case is proved in a similar way.
First observe that for $j<p$ the matrices $(F_z)_j:=\Phi^{j}(F_z)$ and $F_j:=\Phi^{j}(F)$ are both equal to $\diag\,(F_{1,1}, \ldots, F_{j,j})$.
Also for $j \geq q$ the matrices $(F_z)_j$ and $F_j$ have the same eigenvalues.
This is because
the eigenvalues of this $2 \times 2$ matrix
\begin{displaymath}\left[
\begin{array}{cc}
\frac{(v_i +|z|^2 v_k)}{Z} & \frac{\bar{z}(v_i - v_k)}{Z}\\
\frac{z(v_i - v_k)}{Z} & \frac{(v_k +|z|^2 v_i)}{Z}
\end{array} \right],
\end{displaymath}
where $Z=\sqrt{1+|z|^2}$, are $v_i$ and $v_k$.
Therefore, for $j <p$ or $j \geq q$, we have 
\begin{equation}\label{FaequalF}
\forall_{F_z \in S}\,\,\, \lambda^{(j)}_m(F_z)=\lambda^{(j)}_m(F),
\end{equation}
for any $m=1, \ldots, n-j$.
Denote by
$\rho(|z|)=\frac{(v_i +|z|^2 v_{k})}{Z}$. While $a$ goes to $\infty$, $\rho$ decreases its value from $v_i$ to $v_{k}$.
\begin{lemma}\label{identifyingedges}
 For $z$ such that $v_{i}>\frac{(v_i +|z|^2 v_{k})}{Z}=\rho(|z|)>v_{i+1}$ 
the point $\Lambda(F_z)$ is in the interior of an edge of $\cc{P}$.
\end{lemma}
\begin{proof}
Let $m$ be such that $$ \lambda^{(q-1)}_m(F_z)=v_i >\rho(|z|)= \lambda^{(q-1)}_{m+1}(F_z).$$
We will show that for any
$(j,l)\neq(q-1,m)$, $j=1,\ldots,n-1$, $l=1, \ldots, j$, we have that
$$ \lambda^{(j)}_{l}(F_z)=\lambda^{(j+1)}_{l}(F_z)\textrm{   or   }\lambda^{(j)}_{l}(F_z)=\lambda^{(j+1)}_{l+1}(F_z),$$
and use the Lemma \ref{edge}.
The matrix $(F_z)_q:=\Phi^{q}(F_z)$ is diagonal, thus, repeating the proof of Lemma \ref{fixedptvertex} for $(F_z)_q$,
we can show that the above claim
holds for $j<q-1$ and any $l$.
Also, for $j \geq q$ the claim holds, due to equations (\ref{FaequalF}) and Lemma \ref{fixedptvertex}.
Thus, for $j \neq q-1$ and any $l$, the function 
$\lambda^{(j)}_{l}$ is equal at $F_z$ to its lower or upper bound.
The only hard case is when $j=q-1$.
Notice that 
$$spectrum((F_z)_{q})=spectrum ((F_z)_{q-1}) \cup \{v_i, v_{k}\} \setminus \{\rho(|z|)\} .$$
The Figure \ref{figedgelemma}  presents sequences of ordered eigenvalues of $(F_z)_{q-1}$ and $(F_z)_q$.
\begin{figure}[h]
\label{figedgelemma}
	\centering
		\includegraphics[width=1.0\textwidth]{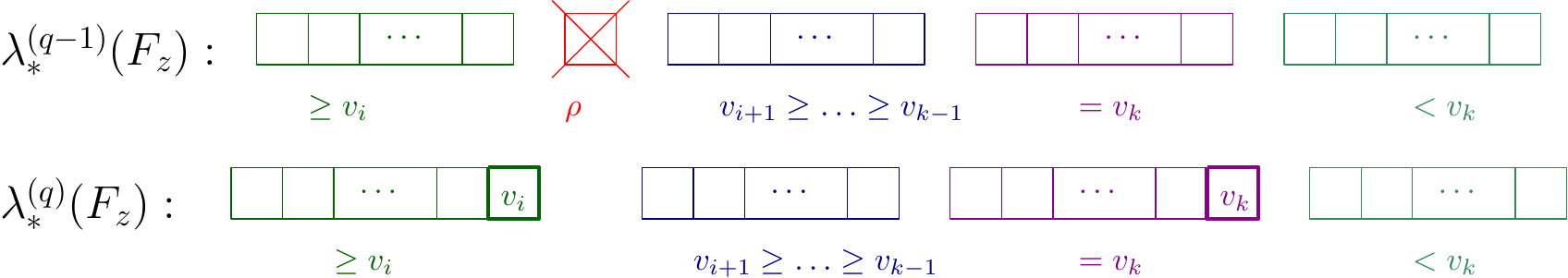}
	\caption{Eigenvalues of $(F_z)_{q-1}$ and $(F_z)_q$.}
\end{figure}
This presentation helps to note that
$$\forall_{t\neq m,\,}\,\lambda^{(q-1)}_{t}(F_z)\geq v_k \,\,\Rightarrow \lambda^{(q-1)}_{t}(F_z)=\lambda^{(q)}_{t}(F_z),$$
$$\forall_{t\neq m,\,}\,\lambda^{(q-1)}_{t}(F_z)< v_k \,\,\Rightarrow \lambda^{(q-1)}_{t}(F_z)=\lambda^{(q)}_{t+1}(F_z).$$
All eigenvalues of $(F_z)_q$ are equal to some element of the set $\{v_1, \ldots, v_{n-s}\}$. Therefore 
$\lambda^{(q-1)}_m(F_z)=\rho(|z|) \in (v_{i+1},v_i)$ is not equal to $\lambda^{(q)}_m(F_z)$ nor  $\lambda^{(q)}_{m+1}(F_z)$.
\end{proof}
\begin{lemma} If $v_{i+1}\in \{\lambda^{(q-1)}_{1}(F_z),\ldots, \lambda^{(q-1)}_{q}(F_z)\}$, then
$\Lambda(\,\{F_z\,|\,\rho(|z|)=v_{i+1}\}\,)$ is a vertex of $\cc{P}$. 
\end{lemma}
 \noindent In particular, if $k=i+1$ then $\Lambda(\,\{F_z\,|\,\rho(|z|)=v_{k}\}\,)$ is a vertex of $\cc{P}$.
\begin{proof}
Similarly to the proof of Lemma \ref{identifyingedges}, we show that for
$(j,l)\neq(q-1,m)$, $j=1,\ldots,n-1$, $l=1, \ldots, j$,
the function $ \lambda^{(j)}_{l}$ at $F_z$ is equal to its lower or upper bound (again use Figure \ref{figedgelemma}).
However this time $\lambda^{(q-1)}_{m}(F_z)=\rho(|z|)=v_{i+1}=\lambda^{(q)}_{m+1}(F_z)$.
We use Lemma \ref{vertex} to deduce that $\Lambda(\,\{F_z\,|\,\rho(|z|)=v_{i+1}\}\,)$ is a vertex of $\cc{P}$. 
\end{proof}
%-------------------------------
In this way we found an edge, or a subset of an edge, of $\cc{P}$ starting from $V_F$. 
Now we need to compute it's length relative to the length of the isotropy weight spanning this edge.
Notice that the projection $pr$ maps the weights of $T^D$ action to the weights of $T^n$ action.
If $e=c_l \eta_l$ is the edge of $\cc{P}$, then $pr(e)=c_l pr(\eta_l)$ is the part of the corresponding edge $e'$ of $\cc{Q}_1$
from $\mu(F)$.
The weight $pr(\eta_l)$ is the negative root $-e_{pp}+e_{qq}$.
We will denote $\widetilde{Z}:=\{F_z\,|\,\rho(|z|)=v_{i+1}\}$ and $\widetilde{V}:=\Lambda(\widetilde{Z})$, 
regardless of the fact if it is a vertex or an interior point of and edge in $\cc{P}$.
Notice that $\widetilde{V}$, has values of $\Lambda$ 
that are different from those of $F$
in exactly $(q-p)$ places. Precisely, for every $p \leq j <q$, there is exactly one $s$ such that 
$ \lambda^{(j)}_s(F)=v_i$ while $ \lambda^{(j)}_s(\widetilde{Z})=v_{i+1}$.
Recall from section \ref{action} that the $k-th$ coordinate of $pr(\{\lambda^{(*)}_*\})$ is given by
$$ (\,pr(\{\lambda^{(*)}_*\})\,)_k=\sum_{s=1}^k\lambda^{(k)}_s-\sum_{s=1}^{k-1}\lambda^{(k-1)}_s $$
for $k>1$ and is equal to $\lambda^{(1)}_1$ for $k=1$.
Therefore $\mu(F)=pr(\Lambda(F))$ and $\mu(\widetilde{Z})=pr(\Lambda(\widetilde{Z}))$ differ only at $p$-th and $q$-th coordinates:
\begin{align*}
(\,pr(\Lambda(F))\,)_p&=\sum_{s=1}^p\lambda^{(p)}_s(F)-\sum_{s=1}^{p-1}\lambda^{(p-1)}_s(F)\\
&=\sum_{s=1}^p\lambda^{(p)}_s(\widetilde{Z})+v_i-v_{i+1}-\sum_{s=1}^{p-1}\lambda^{(p-1)}_s(\widetilde{Z})=
(\,pr(\Lambda(\widetilde{Z}))\,)_p+v_i-v_{i+1}\\
\end{align*}
\begin{align*}
(\,pr(\Lambda(F))\,)_q&=\sum_{s=1}^q\lambda^{(q)}_s(F)-\sum_{s=1}^{q-1}\lambda^{(q-1)}_s(F)\\
&=\sum_{s=1}^q\lambda^{(q)}_s(\widetilde{Z})+v_i-v_{i+1}-(\,\sum_{s=1}^{q-1}\lambda^{(q-1)}_s(\widetilde{Z})+v_i-v_{i+1}\,)\\
&=
(\,pr(\Lambda(\widetilde{Z}))\,)_q-(v_i-v_{i+1})
\end{align*}
Thus
$$\overline{\mu(F)\,\mu(\widetilde{Z})}=(v_i-v_{i+1})(-e_{pp}+e_{qq}),$$
and the edge $e$ of $\cc{P}$ is at least $(v_i-v_{i+1})$ multiple of the weight spanning it.\\
\indent In case where $v_k >v_i$ we would proof in an analogous way that the edge joining $F$ and $F'$ has the lattice length (w.r.t. weight lattice) at least $(v_{i-1}-v_i)$, 
as $\rho(|z|)$ would be increasing its value from $v_i$ to $v_k$.\\
\indent Notice that different pairs of $p$ and $q$ (such that $F_{pp} \neq F_{qq}$) give different edges.
This follows, for example, from the fact that
for $j<p$ or $j \geq q$, we have 
$ \lambda^{(j)}_s(F_z)=\lambda^{(j)}_s(F).$
Therefore we found $D$ edges starting from $V_F$. The Proposition \ref{Dedges} gives that these must be all the edges.\\
\indent
Now suppose that $m$ is the index such that the minimum of $\{v_i - v_{i+1}\,|\, i=1, \ldots, s\}$ is equal to $v_m-v_{m+1}$.
There are indices $p<q$ such that $F_{p,p}=v_m$ and $F_{q,q}=v_{m+1}$, or $F_{p,p}=v_{m+1}$ and $F_{q,q}=v_m$.
Let $F'$ be the diagonal matrix obtained from $F$ by switching $p$-th and $q$-th entry.
Then $\widetilde{Z}=F'$, $\widetilde{V}=\Lambda(F')$ and the edge of $\cc{P}$ between these two vertices is 
exactly $(v_m-v_{m+1})$ multiple of the weight spanning it.
\end{proof}
Proposition \ref{edgelengths} together with Proposition \ref{embed} give the proof of the Main Theorem, as explained in the beginning of this Section.
\end{proof}
%----------------------------------------------------------------------------------------------------
\section{Low-dimensional examples.}\label{summary}
In this section we summarize what is known about Gromov width of $U(n)$ coadjoint orbits.
The table below presents low dimensional examples for which it was proved that lower bound of Gromov width is as expected: the minimum of $\lambda_j-\lambda_j$ over $\lambda_i>\lambda_j$. The table also specifies 
if this fact follows directly from our Main Theorem; if it requires Remark \ref{moresmoothness}; or if it was proved using different methods. 
Generic $U(1)$ orbits, and degenerate $U(2)$ orbits are just points, so their Gromov width is $0$. Gromov width of generic orbits satisfying some integrality conditions was
already calculated by Zoghi in \cite{Z}.\\
\begin{tabular}{c| p{5cm}||c|c|p{4.4cm}}
\hline
n & $\lambda$ & Thm \ref{main}&  Rem. \ref{moresmoothness}& Other \\ \hline 
\hline
2&  generic $\leadsto$ sphere& $\sqrt{}$ &&Delzant Thm; also \cite{Z}\\
& degenerate $\leadsto$points&&&\\ \hline
3& any $\lambda$& $\sqrt{}$ && generic - proved in \cite{Z}\\ \hline
4& $(\lambda_1,\lambda_1,\lambda_2,\lambda_2)$ $\leadsto$ complex Grassmannian of $2$-planes in $\bb{C}^4$&$-$& $\sqrt{}$& Karshon and Tolman, \cite[Theorem 1]{KT}\\ \hline
4&other $\lambda$ & $\sqrt{}$ && generic - proved in \cite{Z}\\ \hline
5& $\begin{cases}
     (\lambda_1, \lambda_1, \lambda_2, \lambda_2, \lambda_3)\\
(\lambda_1, \lambda_1, \lambda_2, \lambda_2, \lambda_2)\\
(\lambda_1,\lambda_1, \lambda_1, \lambda_2, \lambda_2)\end{cases}$ &$-$& $\sqrt{}$& \\  \hline
5& other $\lambda$ & $\sqrt{}$ &&generic - proved in \cite{Z}\\ 
\hline
6&$(\lambda_1, \lambda_1, \lambda_2, \lambda_2, \lambda_3, \lambda_3)$&$-$&$-$&$-$\\ \hline
\end{tabular}\\

In the case of $n=6$, there is already an orbit for which we still don't have even the lower bound of the Gromov width.
Namely $(\lambda_1, \lambda_1, \lambda_2, \lambda_2, \lambda_3, \lambda_3)$. For all the other orbits, the lower bound or even exact Gromov width is proved in 
Theorem \ref{main} together with Remark \ref{moresmoothness}, or in \cite{KT}, or \cite{Z}.\\
%----------------------------------------------------------------
%---------------------------------------------------------------
%----------------------------------------------------------------
%---------------------------------------------------------------
%----------------------------------------------------------------
%---------------------------------------------------------------
%----------------------------------------------------------------
%---------------------------------------------------------------
%----------------------------------------------------------------
%---------------------------------------------------------------

\end{document}